\documentclass[a4paper,11pt]{article}
\usepackage[T1]{fontenc}
\usepackage[latin1]{inputenc}

\usepackage{amsmath,amsfonts}
\usepackage{amsthm}
\usepackage{amssymb}
\usepackage[USenglish]{babel}
\usepackage{amsfonts}
\usepackage{enumerate}
\usepackage{verbatim}
\usepackage{graphicx}
\usepackage{todonotes}
\usepackage[left=3cm, right=3cm, top=3cm, bottom=3cm]{geometry}
\usepackage{tikz}
\usepackage{csquotes}
\usepackage{mathtools}
\usepackage{mathscinet} 
\usepackage{tabularx}
\usepackage{scrextend}
\usepackage{float}

\usepackage[final, pdftex, pdfpagelabels, pdfstartview = {FitH}, bookmarks, colorlinks, plainpages = false, linktoc=all, linkcolor=black,            citecolor=blue, urlcolor=blue, filecolor=black]{hyperref}

\usetikzlibrary{matrix,arrows,decorations.pathmorphing,decorations.pathreplacing,decorations.markings}

\theoremstyle{plain}
\newtheorem{theorem}{Theorem}[section]
\newtheorem{lemma}[theorem]{Lemma}

\newtheorem{cor}[theorem]{Corollary}

\theoremstyle{definition}
\newtheorem{definition}[theorem]{Definition}

\newtheorem{remark}[theorem]{Remark}
\newtheorem{example}[theorem]{Example}

\theoremstyle{remark}

\renewcommand{\tilde}{\widetilde}
\renewcommand{\bar}{\overline}
\newcommand{\bbC}{\mathbb{C}}
\newcommand{\bbD}{\mathbb{D}}

\newcommand{\bbK}{\mathbb{K}}
\newcommand{\bbN}{\mathbb{N}}

\newcommand{\bbR}{\mathbb{R}}
\newcommand{\bbZ}{\mathbb{Z}}

\newcommand{\raw}{\rightarrow}
\newcommand{\xra}{\xrightarrow}
\newcommand{\cug}{\subseteq}

\newcommand{\eps}{\varepsilon}

\newcommand{\undx}{{\underline{x}}}
\newcommand{\undw}{{\underline{w}}}
\newcommand{\undH}{{\underline{\mathbf{H}}}}

\newcommand{\calC}{\mathcal{C}}

\newcommand{\calH}{\mathcal{H}}

\newcommand{\calK}{\mathcal{K}}

\newcommand{\calO}{\mathcal{O}}

\newcommand{\frh}{\mathfrak{h}}

\newcommand{\frB}{\mathfrak{B}}

\newcommand{\sbim}{{\mathbb{S}Bim}}

\newcommand{\isom}{\xrightarrow{\sim}}

\newcommand{\bfH}{\mathbf{H}}
\newcommand{\bfP}{\mathbf{P}}

\newcommand{\cus}{\stackrel{\oplus}{\subseteq}}

\DeclareMathOperator{\cha}{ch}
\DeclareMathOperator{\chara}{char}

\DeclareMathOperator{\End}{End}

\DeclareMathOperator{\grrk}{grrk}

\DeclareMathOperator{\Hom}{Hom}

\DeclareMathOperator{\res}{res}

\DeclareMathOperator{\spa}{span}
\DeclareMathOperator{\supp}{supp}
\DeclareMathOperator{\Sym}{Sym}

\DeclareMathOperator{\cone}{cone}

\title{Singular Rouquier Complexes}

\author{Leonardo Patimo}

\begin{document}

\maketitle

\begin{abstract}
We generalize the construction of Rouquier complexes to the setting of singular Soergel bimodules by taking minimal complexes of restricted Rouquier complexes. We show that  they retain many of the properties of ordinary Rouquier complexes: they are $\Delta$-split, they satisfy a vanishing formula and, when Soergel's conjecture holds they are perverse. As an application, we establish Hodge theory for singular Soergel bimodules. 
\end{abstract}

\section{Introduction}

Consider  a complex reductive algebraic group $G$ with Borel subgroup $B$ and Weyl group $W$. The category of $B$-equivariant parity sheaves on the flag variety $X=G/B$ provides a categorification of the Hecke algebra $\calH$ of $W$. 
Soergel \cite{S1, S4} defined an alternative categorification of the Hecke algebra $\calH$ via certain graded bimodules over $R=\Sym^\bullet(\frh^*)$, where $\frh^*$ is a (well-behaved) representation of $W$. 
A major advantage of using Soergel bimodules is that their construction is completely algebraic, in particular their definition makes sense for an arbitrary Coxeter group $W$.

The situation is very similar when we consider 
a parabolic subgroup $P$ of $G$ containing $B$ and the partial flag variety $G/P$.
Let $I$ be the subset of the simple reflections $S\subset W$ corresponding to $P$. Let $W_I$ denote the subgroup of $W$ generated by $I$. Then, $B$-equivariant parity sheaves on $G/P$ categorify of the left ideal $\calH^I:=\calH \undH_I$ of the Hecke algebra $\calH$, where $\undH_I\in \calH$ is the Kazhdan-Lusztig basis element corresponding to the longest element in $W_I$.
In this case an algebraic replacement is provided by singular Soergel bimodules, introduced by Williamson in \cite{W4}. Singular Soergel bimodules are graded $(R,R^I)$ bimodules, where $R^I$ denotes the subring of $W_I$-invariants in $W$. The construction of singular Soergel bimodules is algebraic and works for any Coxeter group $W$ and any subset $I\subset S$ such that $W_I$ is finite.
The indecomposable Soergel bimodules $B_w^I$ are parameterized (up to grading shifts)  by elements $w\in W^I$, where $W^I$ is the set of  elements of $W$ which are minimal in their right $W_I$-coset.

For a Coxeter group $W$, let $B_W$ denote the corresponding Artin braid group. 
In \cite{Ro}, Rouquier introduced, inside the homotopy category of Soergel bimodules, a categorification of $B_W$: the $2$-braid group $\frB_W$. Let us briefly recall its construction. 
 For any element $s\in S$ let $B_s=R\otimes_{R^s}R(1)$ be the corresponding indecomposable Soergel bimodule and consider the complexes
\[F_s:= [0\raw B_s \raw R(1)\raw 0]\]
\[E_s:= [0\raw R(-1) \raw B_s\raw 0]\]
Notice that $F_s$ and $E_s$ are inverse to each other with respect to the tensor product operation, so we can also write $E_s=(F_s)^{-1}$. 
Then the objects in $\frB_W$ are the complexes that can be obtained as products of $F_s$ and of $E_s$, for $s\in S$.

If $W$ is a finite group then $\frB_W$ is a faithful categorification of $B_W$ \cite{KS,BT,Je}: if to any $w=s_1^{\eps_1}s_2^{\eps_2}\ldots s_k^{\eps_k}\in B_W$ (where $\eps_i =\pm 1$) we associate the complex
\[F_w := (F_{s_1})^{\eps_1}(F_{s_2})^{\eps_2} \ldots (F_{s_k})^{\eps_k}\in \frB_W\] 
then, for $v,w\in B_w$, we have $F_w\cong F_v$ if and only if $w=v$.

Any elements of the Coxeter group $W$ has two distinguished lifts to $B_W$, and hence to $\frB_W$. 
If $w=s_1s_2\ldots s_k\in W$ we define  $F_{s_1}F_{s_2}\ldots F_{s_k}$ to be the positive lift and  $E_{s_1}E_{s_2}\ldots E_{s_k}$ to be the negative lift of $w$ in $\frB_W$.
Let $F_w$ be the minimal complex of $F_{s_1}F_{s_2}\ldots F_{s_k}$, i.e. $F_w$ is the complex in $\calC^b(\sbim)$ obtained by removing all the contractible summands from $F_{s_1}F_{s_2}\ldots F_{s_k}$. Similarly, let $E_w$ be the minimal complex of $E_{s_1}E_{s_2}\ldots E_{s_k}$.
The complexes $F_w$ and $E_w$ are called Rouquier complexes.

One can easily repeat Rouquier's construction in the singular Soergel bimodules world by restricting  a complex of $(R,R)$-bimodules to a complex of $(R,R^I)$-bimodules. 
For any $w\in W^I$ we define the \emph{singular Rouquier complex} $F_w^I$ to be the minimal complex of $\res_{R,R}^{R,R^I} (F_w)$ in the category of complexes of $I$-singular Soergel bimodules $\calC^b(\sbim^I)$. 
We show that \emph{singular 2-braid group} retains some of the important property of the $2$-braid group. 

In \cite{LW}, Libedinsky and Williamson showed that the $2$-braid groups have standard and costandard objects. More precisely, they show that we have  the following vanishing property:
\begin{equation}\label{homvanintro}
\Hom(F_w,E_v[i])=\begin{cases}k & \text{if }v=w\text{ and }i=0 \\ 0 & \text{otherwise}. \end{cases}\end{equation}
(If $W$ is a Weyl group and $k=\bbC$, this statement is equivalent to the existence of standard and costandard object in category $\calO$.)
The main result of this paper is the generalization of the results in \cite{LW} to singular Rouquier complexes. In particular, we prove the singular version of \eqref{homvanintro}:
\[ \Hom(F_w^I,E_v^I[i])=\begin{cases}k & \text{if }v=w\text{ and }i=0 \\ 0 & \text{otherwise}. \end{cases}\]

 It follows that also singular $2$-braid groups  have standard and costandard objects.
We discuss now two applications of this generalization.

\begin{itemize}
	\item In \cite{Pat3} we
	restricts ourselves to the case of Grassmannian, i.e. we consider the symmetric group $W=S_n$ and $W_I$ is a maximal parabolic subgroup. In this setting, summands in singular Rouquier complexes can be understood using the combinatorics of Dyck partitions. 
	A careful study of the first two terms in singular Rouquier complexes allows us to deduce some crucial relations involving maps of degree one, and  these relations allow us to  explicitly construct bases of the morphisms spaces between singular Soergel bimodules. In particular, we also obtain bases for the intersection cohomology of Schubert varieties that naturally extend the Schubert basis.
	\item
When Soergel's conjecture holds, for example when we work over the real numbers and we consider the same representation  of $W$ as in \cite[Prop. 2.1]{S4}, then indecomposable Soergel bimodules categorify the Kazhdan-Lusztig basis in the Hecke algebra. 
In this case, 
Rouquier complexes are perverse and
they categorify the inverse Kazhdan-Lusztig polynomials (as in \cite[Remark 6.10]{EW1}). We show that the same is true for singular Rouquier complexes: they are perverse and from the multiplicities of its summands we can reconstruct the inverse parabolic Kazhdan-Lusztig polynomial.

In \cite{EW1} Rouquier complexes are a crucial tool in establishing Hodge theory for Soergel bimodules, and hence in proving Soergel's conjecture. Elias and Williamson's idea is to emulate the geometric proof of the hard Lefschetz theorem and of the bilinear Hodge-Riemann relations. Here the Rouquier complexes have the decisive role of providing a surrogate for the Lefschetz operator. After having shown that singular Rouquier complexes are perverse, it is rather straightforward 
to adapt the arguments in \cite{EW1} to singular Soergel bimodules.  Hence, we obtain a proof of the hard Lefschetz theorem (Theorem \ref{SHL}) and of the Hodge-Riemann bilinear relations (Theorem \ref{SHR}) for singular Soergel bimodules.

 We remark that using the Hodge theory of singular Soergel bimodules we can give a (slightly) different proof of Soergel's conjecture (cf. Remark \ref{newSConj}), which is closer to the  geometric proof of the  decomposition theorem discussed in \cite{dCM3}.
\end{itemize}

\subsection*{Acknowledgments}
The results of this paper are part of my PhD thesis, written at Max Planck Institute for Mathematics in Bonn under the supervision of Geordie Williamson. I wish to warmly thank him for everything he has taught me.
%, in particular for explaining to me why morphisms of degree one generate all morphisms of Soergel bimodules.

\section{Hecke algebra}
We recall some basic notation about Coxeter groups and their Hecke algebras.
Let $(W,S)$ be a Coxeter system. For $s,t\in S$, let $m_{st}$ denote the order of $(st)$. We denote the  length function by $\ell$ and the  Bruhat order by $\leq$.

The Hecke algebra $\calH:=\calH(W,S)$ is the unital associative $\bbZ[v,v^{-1}]$ algebra with generators $\bfH_s$, for $s\in S$, subject to the following relations, for any $s,t\in S$:
\[\overbrace{\bfH_s\bfH_t\ldots}^{m_st}=\overbrace{\bfH_t\bfH_s\ldots}^{m_{st}}\]
\[\bfH_s^2=-(v-v^{-1})\bfH_s+1.\]
For any $x\in W$  the element $\bfH_x$ is defined as  $\bfH_x:=\bfH_{s_1}\ldots \bfH_{s_l}$ where $x=s_1s_2\ldots s_l$ is any reduced expression for $x$. The set $\{\bfH_x\}_{x\in W}$ is a $\bbZ[v,v^{-1}]$-basis of $\calH$, called the \emph{standard basis}.

We denote by $\bar{(-)}:\calH\raw \calH$ the involution defined by  $\bar{\bfH_s}=\bfH_s^{-1}$ and $\bar{v}=v^{-1}$.
For any $x\in W$ the Kazhdan-Lusztig basis element $\undH_x$ is the unique element in $\calH$ such that $\bar{\undH_x}=\undH_x$ and that we can write 
\[\undH_x=\bfH_x+\sum_{y<x}h_{y,x}(v)\bfH_y.\]
 for some polynomials $h_{y,x}(v)\in v\bbZ[v]$.

The polynomials $h_{y,x}(v)$ are called the \emph{Kazhdan-Lusztig polynomials}. The set $\{\undH_x\}_{x\in W}$ is a $\bbZ[v,v^{-1}]$-basis of $\calH$, called the \emph {Kazhdan-Lusztig basis}.

There exists an anti-involution $a$ of $\calH$ defined by $a(\bfH_x) = \bfH_{x^{-1}}$ for $x\in W$ and $a(v)=v$. The trace $\eps$ is the $\bbZ[v,v^{-1}]$-linear map defined by $\eps(\bfH_w) =\delta_{w,id}$. We define a $\bbZ[v, v^{-1}]$-bilinear pairing
\begin{equation}\label{heckepairing}
(-,-) : \calH \times \calH \raw Z[v, v ^{-1}]
\end{equation}
by $(h, h') = \eps(a(h)h' )$.

For a subset $I\subset S$, let $W_I$ be the parabolic subgroup of $W$ generated by $I$. A subset $I\cug S$ is said \emph{finitary} if the group $W_I$ is finite. We denote by $W^I$ the set of right $I$-minimal elements, i.e. the set of elements $x\in W$ such that  $xs\geq x$ for all $s\in I$.

Let $q:W\raw W/W_I$ denote the projection map. For $y\in W/W_I$  we denote by $y_-$ the minimal element in the coset $y$. The bijection $W^I\cong W/W_I$ induces a partial order on $W/W_I$ by restricting the Bruhat order of $W$, i.e. for $y,z\in W/W_I$ we say $y\leq z$ if and only if $y_-\leq z_-$. The projection $q$ is a strict morphism of posets:
\begin{lemma}[{\cite[Lemma 2.2]{Dou}}]\label{poset}
	Let $w\geq v$ in $W$. Then $q(w)\geq q(v)$.
\end{lemma}

 Let $I$ be finitary and let $w_I$ be the longest element in $W_I$.
 We denote by $\tilde{\pi}(I)$ be the \emph{Poincar\'e polynomial} of $W_I$, i.e.
 \[\tilde{\pi}(I):=\sum_{w\in W_I}v^{2\ell(w)}.\]  We define 
\[\undH_I:=\undH_{w_I}=\sum_{x\in W_I}v^{\ell(w_I)-\ell(x)}\bfH_x.\] 

Consider the left ideal $\calH^I:=\calH\undH_I$ of $\calH$.
For $x\in W^I$ we define $\bfH_x^I=\bfH_x\undH_I$.
 The Kazhdan-Lusztig basis element $\undH_y$ belongs to $\calH^I$ if and only if $y$ is maximal in its right $W_I$-coset. Thus, for $x\in W^I$, we can define $\undH_x^I=\undH_{xw_I}$. The set $\{\undH_x^I\}_{x\in W^I}$ forms a $\bbZ[v,v^{-1}]$-basis of $\calH^I$, called the \emph{$I$-parabolic Kazhdan-Lusztig basis} of $\calH^I$. For any $x\in W^I$ we can write
\[\undH_x^I=\bfH_x^I+\sum_{W^I\ni y<x}h_{y,x}^I(v)\bfH_y^I.\]

The polynomials $h_{y,x}^I(v)$ are called the \emph{$I$-parabolic Kazhdan-Lusztig polynomials} and are related to the ordinary Kazhdan-Lusztig polynomials by the formula $h_{y,x}^I(v)=h_{yw_I,xw_I}(v)$.

\section{One-sided Singular Soergel Bimodules}

The main reference for this section is \cite{W4}.
We fix a field $\bbK$ and a reflection faithful representation $\frh^*$ of $W$ over $\bbK$ (in the sense of \cite{S4}).
Let $R$ denote the polynomial ring $\Sym_\bbK(\frh^*)$.
We regard $R$ as a graded ring by setting $\deg(\alpha)=2$ for any $\alpha\in \frh^*$.

We fix now  a  finitary subset $I\cug S$. We use the abbreviations $(\frh^*)^I:=(\frh^*)^{W_I}$ and $R^I:=R^{W_I}$ to denote the corresponding subspaces of $W_I$-invariant.
We work in the category of graded $(R,R^I)$-bimodules.
We denote by $(1)$ the grading shift on graded bimodules.
If $J\cug I$, and $B$ is a graded $(R,R^J)$-bimodule $B$ we denote by $B_{I}$ the restriction of $B$ to a graded $(R,R^I)$-bimodule.\label{s:IB}

We make the following assumption: the ring $R$ regarded as a $R^I$-module is free of graded rank $\tilde{\pi}(I)$. 
 This is always the case if $\chara(\bbK)=0$. If $\chara(\bbK)=p$ and $\frh^*$ is the representation obtained by extending scalar on the action of $W$ on the weight lattice, the assumption above is satisfied if $p$ is not a torsion prime for $W$ (cf. \cite[Remark 4.1.2]{W4}).

For $s\in S$ let $B_s:=R\otimes_{R^s}R(1)$.
For any sequence of simple reflections $\undw=(s_1,\ldots,s_k)$  we consider the corresponding Bott-Samelson bimodule
\[BS(\undw):=B_{s_1}\otimes_R B_{s_2}\otimes_R \ldots \otimes_R B_{s_k}.\]

\begin{definition}
The category of $I$-singular Soergel bimodules $\sbim^I$ is the smallest full subcategory of graded $(R,R^I)$-bimodules which contains all Bott-Samelson bimodules $BS(\undw)_I$ and which is closed under direct sums, grading shifts and taking direct summands.

Morphisms in $\sbim^I$ are the morphisms of graded $(R,R^I)$-bimodules of degree $0$ and are denoted by $\Hom(-,-)$.

If $I=\emptyset$ then $\sbim^\emptyset$ is simply denoted by $\sbim$ and called the category of Soergel bimodules.
\end{definition}

For any $i\in \bbZ$ we set $\Hom^i(B,B')=\Hom(B,B'(i))$ and
\[\Hom^\bullet(B,B')=\bigoplus_{i\in \bbZ} \Hom(B, B' (i)).\]

There is a duality functor $\bbD B=\Hom_{R-}^\bullet(B,R)$ on $\sbim^I$. The $(R,R^I)$-bimodule structure on $\bbD B$ is given by 
\[(rfr')(b)=f(rbr')=r f(br')\qquad \text{ for any }f\in \bbD B, b\in B, r\in R, r'\in R^I.\]

\begin{theorem}[Soergel-Williamson categorification theorem {\cite[Theorem 1]{W4}}]\label{SWcat}
	There exists a bijection 
	\[
 W^I  \stackrel{1:1}{\longleftrightarrow}  \left\{
	\begin{array}{c}\text{indecomposable self-dual}\\
		I\text{-singular Soergel bimodules}\\ 
		\text{up to isomorphism}
		\end{array}
	\right\}.\]
We denote by $B_x^I$ the indecomposable self-dual bimodule corresponding to $x$.
Every indecomposable $I$-singular Soergel bimodule is isomorphic up to a shift to some $B_x^I$.

Let $x=s_1s_2\ldots s_k$ be a reduced expression for $x\in W^I$. Then $B_x^I$ is the unique direct summand of $BS(\underline{s_1s_2\ldots s_k})_I$ which is not a direct summand of any Bott-Samelson bimodule of smaller length. 
%Equivalently, $B_x^I$ is the unique direct summand of $B_{x,I}:=(B_x)_I$ which is not a direct summand of $B_{y,I}$ for any $y$ such that $\ell(y)<\ell(x)$.\label{s:IBx}
\end{theorem}

Given two bimodules $B_1^I,B_2^I\in \sbim^I$ and $x\in W^I$, consider the subspace 
\[\Hom^\bullet_{< x}(B_1^I,B_2^I)\cug \Hom^\bullet(B_1^I,B_2^I)\] generated by all the maps $\varphi:B_1^I\raw B_2^I(k)$ which factor through $B_1^I\raw B_y^I(k')\raw B_2^I(k)$ for some $y< x$ and $k'\in \bbZ$.
Let \[ \Hom^\bullet_{\not< x}(B_1^I,B_2^I):=\Hom^\bullet(B_1^I,B_2^I)/\Hom^\bullet_{< x}(B_1^I,B_2^I).\]
Both

 Let $[\sbim^I]$ denote the split Grothendieck group of $\sbim^I$. We regard it as a $\bbZ[v,v^{-1}]$-module via $v\cdot [B] = [B(1)]$.
 If $V=\bigoplus_{i\in \bbZ} R(-i)^{m_i}$ is a graded free $R$-module we define the \emph{graded rank} of $V$ as:
 \[\grrk (V) := \sum_{i\in\bbZ}m_i v^{i}.\]
 
The \emph{character map} is a morphism of $\bbZ[v,v^{-1}]$-modules $\cha: [\sbim^I]\raw \calH^I$ defined by
\[\cha([B^I])=\sum_{x\in W^I}\grrk \Hom^\bullet_{\not< x}(B^I,B_x^I)\bfH_x^I\]
for any $B^I\in \sbim^I$.\footnote{The $R$-module $\Hom^\bullet_{\not< x}(B^I,B_x^I)$ is free: this follows from \cite[Theorem 7.2.2]{W4} and the fact that $\Hom_{\not<x}^\bullet(B^I,B_x^I)\cong \Hom^\bullet_{\not< x}(B^I,R_{x,I}(\ell(x)))$, where $R_{x,I}=(R_x)_I$ is a standard module (cf. Remark \ref{not<x}).} It follows from Theorem \ref{SWcat} that $\cha$ is an isomorphism.
Moreover, the following diagram is commutative:
	 \begin{center} 
	 	\begin{tikzpicture}
	 	\matrix(m)[matrix of math nodes, column sep=50pt, row sep=20pt]
	 	{{[\sbim]\times [\sbim^I]} & {[\sbim^I]}\\
	 		\calH\times\calH^I & \calH^I\\};
	 	\path[->] (m-1-1) edge node[above]{$-\otimes_R-$}(m-1-2)
	 	(m-1-1) edge node[left] {$\cha\times\cha$} (m-2-1)
	 	(m-1-2) edge node[left] {$\cha$} (m-2-2)
	 	(m-2-1) edge node[above] {$m$} (m-2-2);
	 	\end{tikzpicture}
	 \end{center}
	 (here $m$ is the multiplication in $\calH$).
Hence $\sbim^I$ \emph{categorifies} the ideal $\calH^I$ as a module over $\calH$.
We can use the isomorphism $\cha$ to compute the dimension of the space of morphisms in the category $\sbim^I$.

\begin{theorem}[Soergel's Hom Formula for Singular Soergel Bimodules {\cite[Theorem 7.4.1]{W4}}]\label{SHF}
Let $B_1,B_2\in \sbim^I$. Then $\Hom^\bullet(B_1,B_2)$ is a free graded left $R$-module and 
\[\grrk \Hom_{R\otimes R^I}^\bullet(B_1,B_2) = \frac{1}{\tilde{\pi}(I)}( \bar{\cha(B_1)},\cha(B_2) ), \]
where $(-,-)$ is the pairing in the Hecke algebra defined in \eqref{heckepairing}.
\end{theorem}

We can identify $R\otimes_\bbR R^I$ with the ring of regular functions on $\frh\times (\frh/W_I)$. 
Hence a Soergel bimodule $B^I\in \sbim^I$ can be thought as a quasi-coherent sheaf on $\frh\times (\frh/W_I)$.
The inclusion $R\otimes_\bbR R^I\hookrightarrow R\otimes_\bbR R$ corresponds to the projection map $\pi:\frh\times \frh\raw  \frh\times (\frh/W_I)$. 

For $x \in W$ we denote the twisted graph of $x$ by $Gr(x)$, that is
\[Gr(x) = \{(x\cdot \lambda, \lambda) | \lambda \in \frh\} \cug \frh \times \frh.\]
If $C\cug W$, let $Gr(C)=\bigcup_{x\in C}Gr(x)$.
For a coset $y\in W/W_I$ let $Gr^I(y):=\pi(Gr(y))$. Notice that $Gr^I(y)=\pi(Gr(\tilde{y}))$ for any $\tilde{y}\in y$.  Similarly, if $C\cug W/W_I$, let $Gr^I(C):=\bigcup_{p\in C}Gr^I(p)$.

 The support of every Soergel bimodule $B^I$ is contained in $Gr(W/W_I)$. For $C\cug W/W_I$ we define
\[\Gamma^I_CB=\{b\in B\mid \supp{b}\cug Gr^I(C)\}.\]
We will simply write $\Gamma_C$ for $\Gamma^\emptyset_C$.
For any $B\in \sbim$ and any $C\cug W/W_I$ we have by \cite[Prop 6.1.6]{W4}
\begin{equation}\label{support}
(\Gamma_{q^{-1}(C)}B)_I=\Gamma^I_C (B_I).
\end{equation}

\begin{remark}\label{alt-def}
	We would like to draw attention to  few slight differences with the definitions given in \cite{W4}.
	Our definition of the duality functor $\bbD$ contains a different shift, and thus our self-dual indecomposable bimodules $B_x^I$ coincide with $B_x^I(-\ell(w_I))$ in Williamson's notation. 
	The advantage of our definition of $\bbD$ is that it guarantees that the singular Soergel modules $\bar{B_x^I}=\bbK\otimes_R B_x^I$ have symmetric Betti numbers. This is more natural in the geometric setting where these modules can be  obtained as the hypercohomology of the intersection cohomology complex on a Schubert variety. This choice of the shift is particularly convenient when dealing with Hodge theoretic properties (cf. Section \ref{sec:hodge}).

	We point out that with our definition of the duality $\bbD$ we have
	$\cha(B_I)=\cha(B)\undH_{w_I}$  and if $x\in W^I$ we have 
	\[B_x^I\otimes_{R^I}R(\ell(w_I))\cong B_{xw_I}\in \sbim.\]

	Furthermore, the definition of the category $\sbim^I$ given above slightly differs from   \cite{W4}. To show that the two definitions give rise to equivalent categories it is enough to show that for any $x\in W^I$ we can obtain the indecomposable bimodule $B_x^I$ as a direct summand of the restriction $B_I$ of a Soergel bimodule $B\in \sbim$. Let $\undx$ be a reduced expression for $x$. We have  
	\[\cha(BS(\undx)_I)=\undH_{\undx}\undH_{w_I}=H_x^I+\sum_{W^I\ni y<x}\lambda_y H_y^I,\]
	for some $\lambda_y\in \bbZ_{\geq 0}[v,v^{-1}]$.
	As in the proof of \cite[Theorem 7.4.2]{W4} this implies that $B_x^I$ is a direct summand of $BS(\undx)_I$. 
\end{remark}

\section{Singular Rouquier Complexes} 

Let $\calC^b(\sbim^I)$ be the bounded category of complexes of $I$-singular Soergel bimodules and let $\calK^b(\sbim^I)$ be the corresponding bounded homotopy category.

Following the notation of \cite{EW1}, we indicate the homological degree of an object $F\in \calC^b(\sbim^I)$ on the left as follows:
\[F=[\ldots\raw {}^{i-1}F\raw {}^iF\raw {}^{i+1}F\raw \ldots].\]
We denote by $[-]$ the homological shift, so that ${}^i(F[1])={}^{i+1}F$.

For $s\in S$ let $F_s$ denote the complex\footnote{We use here the notation $\overset{0}{-}$ to indicate where the object in  homological degree $0$ is placed.}
\[F_s=[0\raw \overset{0}{B_s}\xra{d_s} R(1)\raw 0]\]
where $d_s$ is the map defined by $f\otimes g\mapsto fg$. Then tensoring with $F_s$ on the left induces an equivalence on the category $\calK^b(\sbim^I)$. In fact, tensoring on the left with  the complex
\begin{equation}\label{Es}
E_s=[0\raw R(-1)\xra{d'_s} \overset{0}{B_s}\raw 0].
\end{equation}  
gives an inverse. Here the map $d'_s$ is defined by $d'_s(1)=c_s=\frac{1}{2}(\alpha_s\otimes 1+1\otimes \alpha_s)$.

Given $x\in W^I$ and any reduced expression $x=s_1\ldots s_k$, we  consider the complex $F_{s_1}\ldots F_{s_k}$. As an object in $\calK^b(\sbim)$, the complex $F_{s_1}\ldots F_{s_k}$  does not depend on the chosen reduced expression \cite[Proposition 9.2]{Ro}. Hence, also  $(F_{s_1}\ldots F_{s_k})_I$ does not depend on the reduced expression as an object in $\calK^b(\sbim^I)$.

We choose $F_x^I\cus (F_{s_1}\ldots F_{s_k})_I$\label{s:IFx} to be the corresponding minimal complex (cf. \cite[\S 6.1]{EW1}), so $F_x^I\cong (F_{s_1}\ldots F_{s_k})_I$ in $\calK^b(\sbim^I)$ and the complex $F_x^I$ does not contain any contractible direct summand.
We call $F_x^I$ a \emph{$I$-singular Rouquier complex}.

Observe that if $F_x\in \calK^b(\sbim)$ is the Rouquier complex for $x$, i.e. if $F_x$ is the minimal complex for $F_{s_1}\ldots F_{s_k}$, then $F_x^I$ can also be obtained as the minimal complex of $F_{x,I}:=(F_x)_I$ in $\calK^b(\sbim^I)$.

\subsection{Singular Rouquier Complexes are \texorpdfstring{$\Delta$}{Delta}-split}\label{Deltasplit}

If $x\in W^I$ we write $\Gamma^I_{\geq x}$ for the functor $\Gamma^I_{\{y\in W^I\mid y\geq x\}}$ on $\sbim^I$. We define similarly $\Gamma^I_{>x}$, $\Gamma^I_{<x}$ and $\Gamma^I_{\leq x}$. Let $\Gamma^I_{\geq x/>x}B:= (\Gamma^I_{\geq x}B)/(\Gamma^I_{>x}B)$ and $\Gamma^I_{\leq x/<x}B:= (\Gamma^I_{\leq x}B)/(\Gamma^I_{<x}B)$.
Recall the projection $q:W\raw W/W_I$. If $y\in W/W_I$ we have $q^{-1}(\geq y)=\{x\in W\mid x\geq y_-\}$, hence 
\[(\Gamma_{\geq (y_-)}B)_I=\Gamma^I_{\geq y} (B_I).\]

We choose an enumeration $y_1,$ $y_2,$ $y_3\ldots$ of $W^I$ refining the Bruhat order on $W^I$ and an enumeration  $w_1,w_2,\ldots w_{|W_I|}$ of $W_I$ refining the Bruhat order of $W_I$. Let  \[z_1=y_1w_1,\;z_2=y_1w_2,\;\ldots,\; z_{|W_I|}=y_{1}w_{|W_I|},\;z_{|W_I|+1}=y_{2}w_1,\;z_{|W_I|+2}=y_{2}w_{2}\ldots .\]
Using Lemma \ref{poset} we can see that also $z_1,z_2,z_3\ldots$ is an enumeration of $W$ which refines the Bruhat order.

We denote by $\Gamma^I_{\geq m}$ the functor $\Gamma^I_{\{y_i: i\geq m\}}$ on $\sbim^I$ and by $\Gamma_{\geq m}$ the functor $\Gamma_{\{z_i : i\geq m\}}$ on $\sbim$. For $l\geq k$, let 
\[\Gamma^I_{\geq k/\geq l}B:=(\Gamma^I_{\geq k}B)/(\Gamma^I_{\geq l}B)\]
and similarly for $\Gamma_{\geq k/\geq l}$, $\Gamma^I_{\leq k/\leq l}$, $\Gamma_{\leq k/\leq l}$.

All the functors above ($\Gamma^I_{\geq x}, \Gamma^I_{\geq x/>x}$, etc.),  extend to functors between the respective homotopy categories, e.g. 
the functor $\Gamma^I_{\geq k/\geq l}$ extends to a functor 
\[\calK^b(\sbim^I)\raw \calK^b(R\text{-Mod-}R^I)\] which we denote again simply by $\Gamma^I_{\geq k/\geq l}$.

For $x\in W$ let $R_x$ denote the corresponding standard bimodule (cf. \cite{EW1}) and for $x\in W^I$ let $R_{x,I}:=(R_x)_I$.
Fix $y=y_mW_I\in W/W_I$ and $x\in W^I$. Let $k=|W_I|(m-1)+1$, so that we have $y_m=z_k$ and $y_{m+1}=z_{k+|W_I|}$.
Then by \eqref{support} and the hin-und-her Lemma for singular Soergel bimodules \cite[Lemma 6.3.2]{W4} we have 
\begin{equation}\label{meh}
\Gamma^I_{\geq y/>y}(F_{x}^I)\cong\Gamma^I_{\geq y/>y}(F_{x,I})\cong \Gamma^I_{\geq m/\geq m+1}(F_{x,I})\cong(\Gamma_{\geq k/\geq {k+|W_I|}}F_x)_I\in\calK^b(R\text{-Mod-}R^I).
\end{equation}

For any $i$ such that $0\leq i\leq |W_I|-1$ we have an exact sequence of complexes of $R$-bimodules.
\begin{equation}\label{doesnotsplit}
0\raw \Gamma_{\geq k/\geq k+i}F_x\raw\Gamma_{\geq k/\geq k+i+1}F_x \raw \Gamma_{\geq k+i/\geq k+i+1}F_x\raw 0.
\end{equation}
Notice that in general a short exact sequence of complexes does not induce a distinguished triangle in $\calK^b(R$-Mod-$R)$.\footnote{for example, for $s\in S$, we have
	\[ 0 \raw F_s/\Gamma_s F_s \raw F_s \raw \Gamma_s F_s\raw 0\]
but $F_s/\Gamma_s F_s\cong 0$ and $F_s\not\cong R_s(-1)\cong \Gamma_s F_s$ in $\calK^b(R$-Mod-$R)$.}  
However, we claim that, after restricting to $\calK^b(R-$Mod-$R^I)$, the sequence \eqref{doesnotsplit} does indeed induce a triangle in $\calK^b(R$-Mod-$R^I)$. 
\begin{lemma} The restriction to $R$-Mod-$R^I$ of the exact sequence of complexes \eqref{doesnotsplit} is termwise split (i.e. every row is split exact).
\end{lemma}
\begin{proof}
	Each term in $\Gamma_{\geq k+i/\geq k+i+1}F_x$ is isomorphic to direct sums of shifts of $R_{y_mw_{i}}$. By induction on $i$, each term in  $\Gamma_{\geq k/\geq k+i}F_x$ can be obtained as an extensions of the standard modules $R_{y_mw_{j}}$, with $j< i$. By \cite[Lemma 6.2.4]{W4}, all the extensions between $R_{y_mw_{i}}$ and $R_{y_mw_{j}}$ with $j\neq i$ become split after restricting to $R$-Mod-$R^I$. It follows that the exact sequence \eqref{doesnotsplit} becomes termwise split after restricting to $R$-Mod-$R^I$.
\end{proof}

Hence, we have the following distinguished triangle
 in $\calK^b(R$-Mod-$R^I)$:
\begin{equation}\label{butatleastisatriangle}
(\Gamma_{\geq k/\geq k+i}F_x)_I\raw(\Gamma_{\geq k/\geq k+i+1}F_x)_I \raw (\Gamma_{\geq k+i/\geq k+i+1}F_x)_I\xra{[1]}.
\end{equation}

We can now prove the singular analogue of \cite[Prop 3.7]{LW}:

\begin{lemma}\label{deltasplit}
	Let $x,y\in W^I$. Then
	\[\Gamma^I_{\geq y,> y}(F_x^I)=\begin{cases}0&\text{if }y\neq x,\\R_{x,I}(-\ell(x))&\text{if }y=x.\end{cases}\]
\end{lemma}

\begin{proof}
Let $m$ be such that $y_m= y$ and let $k$ be with $z_k=y_m$.
First assume  $x\neq y$. Then $x=z_j$ with $j<k$ or $j\geq k+|W_I|$.
For any $i$ with $0\leq i\leq |W_I|-1$,  by \cite[Prop 3.7]{LW}, we have
\[\Gamma_{\geq k+i/\geq k+i+1}F_x \cong 0.\]
Using \eqref{butatleastisatriangle}, 
by induction we obtain $\Gamma_{\geq k/\geq {k+|W_I|}}F_x\cong 0$, and by \eqref{meh} we get 
\[\Gamma^I_{\geq y/>y}(F_x^I)\cong \Gamma^I_{\geq y/>y}(F_{x,I})\cong (\Gamma_{\geq k/\geq {k+|W_I|}}F_x)_I\cong 0\in \calK^b(R\text{-Mod-}R^I).\]

Assume now $x=y$, so that $x=z_k$. Since
\[\Gamma_{\geq k/\geq k+1}F_x\cong R_x(-\ell(x)),\]
using \eqref{butatleastisatriangle} we can show by induction that
$(\Gamma_{\geq k/\geq {k+|W_I|}}F_x)_I\cong R_{x,I}(-\ell(x))$,  and finally by \eqref{meh} we get
\[\Gamma^I_{\geq x/>x}(F_x^I)\cong \Gamma^I_{\geq x/> x}(F_{x,I})\cong (\Gamma_{\geq k/\geq {k+|W_I|}}F_x)_I\cong  R_{x,I}(-\ell(x)).\qedhere\]
\end{proof}

%\begin{remark}\label{HomRC}
%	The exact sequence of right $R$-modules
%	\[0\raw R_s(-1)\raw B_s\raw R(1)\raw 0\]
%	splits, hence we have 
%	$F_s\cong R(-1)$ in $\calK^b(R$-mod$)$. Similarly $(F_{s_1}F_{s_2}\ldots F_{s_k})_I\cong R_I(-k)$ in $\calK^b($Mod-$R^I)$.
%	For $x\in W^I$, let $\bar{F_x^I}=\bbK\otimes_R F_x^I$ in $\calK^b($Mod-$R^I)$. It follows that we have:
%	\[H^i(\bar{F_x^I})=\begin{cases}
%	\bbK(-\ell(x)) &\text{ if }i=0,\\
%	0 &\text{ if }i\neq 0.
%	\end{cases}\]
%\end{remark}

Dually, given  $x\in W^I$ we can define the complexes $E_x^I$ as the minimal complex of $(E_{s_1}E_{s_2}\ldots E_{s_k})_I$, where $s_1s_2\ldots s_k$ is any reduced expression of $x$ (the complex $E_s$ is defined in \eqref{Es}). Similar arguments to those above show that for any $x,y\in W^I$ we have
\[\Gamma^I_{\leq y,< y}(E_x^I)=\begin{cases}0&\text{if }y\neq x,\\R_{x,I}(\ell(x))&\text{if }y=x.\end{cases}\]
%Moreover, $E_x^I\in{}^p\calK^0$ and ${}^{-i}E_x^I\cong {}^iF_x^I$.

As in \cite{LW}, we can define the \emph{augmented singular Rouquier complexes} as 
\begin{eqnarray*}
\tilde{F_x^I}:=\cone(f_x) & \text{where} & f_x: R_{x,I}(-\ell(x))=H^0(F_x^I)\raw F_x\\
\tilde{E_x^I}:=\cone(e_x) & \text{where} & e_x: E_x^I\raw R_{x,I}(\ell(x))=H^0(E_x^I).\\
\end{eqnarray*}

We write $\Hom_\calK(-,-)$ to denote the morphisms in $\calK^b(R$-Mod-$R^I)$.
Combining \cite[Theorem 7.4.1]{W4} and Lemma \ref{deltasplit} we obtain, by the same argument of \cite[Corollary 3.10]{LW}, the following result.

\begin{cor}\label{augmented}
	For any $H\in \calK^b(\sbim^I)$ we have
	\[\Hom_\calK(H,\tilde{E_x^I})=0=\Hom_\calK(\tilde{F_x^I},H).\]
\end{cor}

We also obtain a generalization of \cite[Theorem 1.1]{LW}.

\begin{lemma}\label{deltanabla}
	For any $x,y\in W^I$ and $m\in \bbZ$ we have 
	\[\Hom_{\calK}(F_x^I,E_y^I[m])\cong \begin{cases}R_I & \text{if }x=y\text{ and }m=0,\\ 0& \text{otherwise.}\end{cases}\]
\end{lemma}
\begin{proof}
It follows form Corollary \ref{augmented} that
	\[ \Hom_{\calK}(F_x^I,R_{y,I}(\ell(y))[m])\cong \Hom_{\calK}(F_x^I,E_y^I[m])\cong \Hom_{\calK}(R_{x,I}(-\ell(x)),E_y^I[m]). \]
	
	Notice that all the summands of ${}^iF_x^I$  and of ${}^iE_x^I$ are of the form $B_z^{I}(m_z)$ for some $z\leq x$ and, moreover, we have $z<x$  if $i\neq 0$.
	In particular, if $y\not\leq x$ we have $\Hom({}^iF_x^I,R_{y,I}(\ell(y))=0$ for all $i$, hence \[\Hom_{\calK}(F_x^I,R_{y,I}(\ell(y)[m])=0\] for all $m$.
	Dually, if $x\not\leq y$ we have $\Hom(R_{x,I}(-\ell(x)),{}^iE_{y}^I)=0$ for all $i$, hence \[\Hom_{\calK}(R_{x,I}(-\ell(x)),E_{y}^I[m])=0\] for all $m$.
	It remains to consider the case $x=y$. We have 
	\begin{multline*}
	\Hom_\calK(R_{x,I}(-\ell(x)),E_x^I)\cong \Hom_{\calK}(R_{x,I}(-\ell(x)),{}^0E_x^I)=\\=\Hom_{\calK}(R_{x,I}(-\ell(x)),\Gamma^I_{\leq x/<x}({}^0E_x^I))=\Hom_{\calK}(R_{x,I}(-\ell(x)),R_{x,I}(-\ell(x)))=R_I.\qedhere
	\end{multline*}
\end{proof}

\subsection{Singular Rouquier Complexes and the Support Filtration}

The homological properties of (singular) Rouquier complexes observed in the last section turn out to be useful to understand the support filtration of (singular) Rouquier complexes.

Let $x\in W^I$ and consider a reduced expression $\undx=s_1s_2\ldots s_k$. The bimodule $B_x^I$ is a direct summand of $BS(\undx)_I={}^0(F_{s_1}F_{s_2}\ldots F_{s_k})_I$, but it is not a direct summand of ${}^i(F_{s_1}F_{s_2}\ldots F_{s_k})_I$ for any $i>0$. Hence $B_x^I$ must also be a direct summand of ${}^0F_x^I$. Similarly $B_x^I$ is a direct summand of ${}^0E_x^I$.

\begin{lemma}\label{Rqfact}
	Let $x,y \in W^I$ with $y< x$ and  $m\in\bbZ$. Then every map $B_y^I(m)\xra{\varphi} B_x^I$ factors through ${}^{-1}E_x^I$.
\end{lemma}
\begin{proof}
	After choosing a decomposition ${}^0E_x^I=B_x^I\oplus ({}^0E_x^I)'$, the map $\varphi$ induces a map $\varphi: B_z^I(m)\raw {}^0E_x^I$.	
	By Corollary \ref{augmented} we have an exact sequence
	\[\Hom(B_y^I(m),{}^{-1}E_x^I)\raw \Hom(B_y^I(m), {}^0E_x^I)\raw \Hom(B_y^I(m), R_{x,I}(\ell(x)))\raw 0.\]
	The claim now follows since $\Hom(B_y^I(m), R_{x,I}(\ell(x)))=0$ for $y<x$.	
\end{proof}

\begin{remark}\label{not<x}
	Let now $x,y\in W^I$ be arbitrary. Choose a decomposition ${}^0E_x^I=B_x^I\oplus ({}^0E_x^I)'$ as above. Since $\Hom^\bullet(({}^0E_x^I)',R_{x,I})=0$, by Corollary \ref{augmented} we also have an exact sequence 
	\[ \Hom^\bullet(B_y^I,{}^{-1}E_x^I)\xra{\vartheta} \Hom^\bullet(B_y^I,B_x^I)\raw \Hom^\bullet(B_y^I,R_{x,I}(\ell(x)))\raw 0\]
	
	We claim that the image of the map $\vartheta$ is $\Hom_{<x}^\bullet(B_y^I,B_x^I)$.
	In fact, if a map $B_y^I\raw  B_x^I(k)$ factors  through $B_z^I(k')$ for some $z<x$, then by Lemma \ref{Rqfact} it also factors through ${}^{-1}E_x^I(k)$. In addition, we have
	\[\Hom_{\not <x}^\bullet(B_y^I,B_x^I)\cong \Hom^\bullet(B_y^I,R_{x,I}(\ell(x)))\cong \Hom^\bullet(\Gamma_{\geq x/>x}^IB_y^I,R_{x,I})(-\ell(x))\]	
	where the second isomorphism is \cite[Theorem 7.3.5 (ii)]{W4}. So we can give an equivalent definition of the character map $\cha:[\sbim^I]\raw \calH^I$ via 
	\begin{equation}\label{newch}
	\cha([B^I])= \sum_{x\in W^I}\grrk \Hom^\bullet(B^I,R_{x,I})v^{-\ell(x)}\bfH_x^I=\sum_{x\in W^I}\bar{\grrk(\Gamma_{\geq x/>x}^IB^I)}v^{\ell(x)}\bfH_x^I
	\end{equation}
	where $\bar{(-)}:\bbZ[v,v^{-1}]\raw \bbZ[v,v^{-1}]$ is the automorphism defined by  $\bar{v}=v^{-1}$.
\end{remark}

We can use Lemma \ref{deltasplit} to give a useful characterization of the support filtration. For $x\in W^I$, the elements in $B_x^{I}$ of degree $-\ell(x)$ form a one-dimensional vector space. Let $c_{\text{bot}}\in B_x^I$ be a non-zero element of this vector space.
\begin{lemma}\label{suppochar}
	Let $B^I\in \sbim^I$ and $y\in W^I$. Then 
	\begin{equation}\label{suppo}
	\Gamma_{\leq y}^IB^I=\spa_R\langle \varphi(c_{\text{bot}}) \mid \varphi\in \Hom^\bullet(B_y^I,B^I)\rangle.
	\end{equation}
\end{lemma}
\begin{proof}
	
	For $b\in B_y^I$, we clearly have $\supp \varphi(b) \cug \supp b\cug \{\leq y \}$, hence the inclusion $\supseteq$ in \eqref{suppochar} follows.
	We show now the reverse inclusion. 	
	
	If $y\geq z$ there exists a morphism $\psi:B_y^I\raw B_z^I(\ell(y)-\ell(z))$ such that $\psi(c_{\text{bot}})=c_{\text{bot}}$. So we can replace the RHS in \eqref{suppo} with $\spa_R\langle \varphi(c_{\text{bot}}) \mid \varphi\in \Hom^\bullet(B_z^I,B^I)$ for some $z\leq y\rangle$.
	
	It is enough to show the claim for $B^I$ indecomposable, i.e. $B=B_x^I$ for some $x\in W^I$.
	 Since 
	 \[\Gamma_{\leq y}^I(B_x^I)=\bigcup_{z\leq x \text{ and }z\leq y}\Gamma_{\leq z}^I B_y^I\] it is enough to consider the case $y\leq x$.
	Let $b\in \Gamma_{\leq y}^I(B_x^I)$.
	Consider the singular Rouquier complex $E_x^I$. If $y<x$, from $\Gamma_{\leq y}E_x^I\cong 0$  we deduce that $\Gamma_{\leq y}({}^{-1}E_x)\raw \Gamma_{\leq y}(B_x^I)$ is surjective.
	Moreover, every direct summand in ${}^{-1}E_x$ is of the form $B_z^I(k)$ with $z< x$, so the claim easily follows by induction on $\ell(x)$. 
	
	If $y=x$ we have
	$\Gamma_{\leq x/<x}B_x^I\cong R_{x,I}(\ell(x))$, and it is generated by the image of $c_{\text{bot}}$. Hence for any $b\in B_x^I$ there exists $f\in R$ such that  $b-fc_{\text{bot}} \in \Gamma_{<y}^IB^I_x$. The claim now follows from  the previous case.
\end{proof}

\subsection{Soergel's Conjecture and the Perverse Filtration}

For some of our applications we need  Soergel's conjecture to hold for our representation $\frh^*$. To ensure this, we require that the results of \cite{EW1} are available, i.e. we require that  $\bbK=\bbR$ and assume $\frh^*$ is a reflection faithful representation of $W$ with a good notion of positive roots (cf. \cite[\S 2]{Deo3}). 
Such a representation always exist: see for example the construction given in 
 \cite[Prop 2.1]{S4} or in \cite[Prop 1.1]{Ric}.
By \cite[Theorem 3]{W4}, Soergel's conjecture for Soergel bimodule \cite{EW1} implies the corresponding result for singular Soergel bimodules:
\begin{theorem}\label{SConj}
Assume $\bbK=\bbR$ and $\frh^*$ as above. Then for $x\in W^I$ we have $\cha(B_x^I)=\undH^I_x$.
\end{theorem}

With these assumptions, it follows from Theorem \ref{SHF} that for $x>y$ we have
 \[ \grrk \Hom^\bullet_{\not< y}(B_x^I,B_y^I)=h_{y,x}^I(v)\]
and, as a consequence, for any $x,y\in W^I$
\begin{equation}\label{homvanishing}
\Hom^i(B_x^I,B_y^I)\cong \begin{cases}0&\text{if }i<0,\text{ or }i=0\text{ and }x\neq y\\
\bbR & \text{if }i=0\text{ and }x=y.\end{cases}
\end{equation}

For any bimodule $B^I\in \sbim^I$ we have a (non-canonical) decomposition
\begin{equation}\label{singdec}
B^I=\bigoplus (B_x^I(i))^{\oplus m_{x,i}},
\end{equation}
and we can define the \emph{perverse filtration} $\tau$ on $B^I$ as
\[\tau_{\leq j}B^I=\bigoplus_{i\geq -j} (B_x^I(i))^{\oplus m_{x,i}}.\]
As a consequence of the vanishing of homomorphisms of negative degree \eqref{homvanishing}, the perverse filtration does not depend on the choice of the  decomposition in \eqref{singdec}.

A bimodule $B^I\in \sbim^I$ is said \emph{perverse} if we can write $\cha([B^I])=\sum_{x\in W^I}m_x\undH_x^I$ with $m_x\in \bbZ_{\geq 0}$ or, equivalently, if $\tau_{\leq -1}B^I=0$ and $\tau_{\leq 0}B= B$.

\begin{definition}
	We define ${}^p\calK^{\geq 0}$ to be the full subcategory of $\calK^b(\sbim^I)$ with objects complexes in $\calK^b(\sbim^I)$ which are isomorphic to a complex $F$ satisfying $\tau_{\leq -i-1}{}^iF=0$ for all $i\in \bbZ$.
	
	Similarly, we define ${}^p\calK^{\leq 0}$ to be the full subcategory whose objects are complexes in $\calK^b(\sbim^I)$ which are isomorphic to a complex $F$ satisfying ${}^iF=\tau_{\leq -i}{}^iF$ for all $i\in \bbZ$. 
	
	Let ${}^p\calK^0:={}^p\calK^{\geq 0}\cap {}^p\calK^{\leq 0}$.
\end{definition}

It follows from Theorem \ref{SConj} and Theorem \ref{SHF} that the pair $({}^p\calK^{\leq0},{}^p\calK^{\geq 0})$ defines a non-degenerate $t$-structure on $\calK^b(\sbim^I)$, called the \emph{perverse $t$-structure}. We denote by ${}^p\calK^0$ the \emph{heart} of this $t$-structure. One should regard ${}^p\calK^0$ as the category of equivariant mixed perverse sheaves on the (possibly non-existent) parabolic flag variety associated to $I$.

It is clear that the following statement analogous to \cite[Lemma 6.1]{EW1} holds in the singular setting: for a distinguished triangle
\[F'\raw F\raw F''\xra{[1]}\]
in $\calK^b(\sbim)$,  if $F',F''\in {}^p\calK^{\geq0}$ (resp. $\calK^{\leq 0}$), then $F\in {}^p\calK^{\geq0}$ (resp. $\calK^{\leq 0}$).

\begin{lemma}\label{RC1}
	Given a Rouquier complex $F_x\in \calK^b(\sbim)$, the functor 
	\[F_x\otimes (-):\calK^b(\sbim^I)\raw \calK^b(\sbim^I)\] is left $t$-exact with respect to the perverse $t$-structure, i.e. it restricts to a functor ${}^p\calK^{\geq0}\raw {}^p\calK^{\geq0}$.
\end{lemma}

\begin{proof}
	We can assume $x=s\in S$.
	Since the category ${}^p\calK^{\geq 0}$ is generated under extensions by the objects $B_y^I(m)[n]$, with $y\in W^I$ and $m+n\leq 0$ it is enough to show that $F_sB_y^I\in {}^p\calK^{\geq0}$ for all $y\in W^I$.
	We divide the proof into two cases:
	\begin{enumerate}[i)]
		\item Assume $sxw_I>xw_I$. We have $\displaystyle \cha(B_sB_x^I)=\undH_s\undH_{x}^I=\undH_{sx}^I +\sum_{\substack{z\in W^I\\ z<xs}} m_z\undH_z^I$ with $m_z\in \bbZ_{\geq 0}$. From Theorem \ref{SConj}, we get
		\[B_sB_x^I\cong B_{sx}^I\oplus \bigoplus_{\substack{z\in W^I\\ z<sx}} (B_z^I)^{\oplus m_z}.\]
		and the complex
		\[F_sB_x^I= [0\raw \overset{0}{B_sB_x^I}\raw B_x^I(1)\raw 0]\]
		is manifestly in ${}^p\calK^{\geq 0}$.
		\item Assume $sxw_I<xw_I$. Then we have $\cha(B_sB_x^I)=\undH_s\undH_{x}^I=\undH_s\undH_{xw_I}=(v+v^{-1})\undH_{x}^I$.  Therefore $B_sB_x^I\cong B_x^I(1)\oplus B_x^I(-1)$ and 
		\[F_sB_x^I=[0\raw \overset{0}{B_x^I(1)\oplus B_x^I(-1)}\raw B_x^I(1)\raw 0].\]
		Tensoring with $F_s$ induces an equivalence on the category $\calK^b(\sbim^I)$, and since $B_x^I$ is indecomposable also the complex $F_sB_x^I$ must be indecomposable. Therefore, the map $B_x^I(1)\raw B_x^I(1)$ cannot be trivial, otherwise $\overset{0}{B_x^I(1)}$ would be a non-trivial direct summand of $F_sB_x^I$.
		Since $B_x^I(1)\raw B_x^I(1)$ is non zero, it is an isomorphism by \eqref{homvanishing} and $B_x^I(1)\raw B_x^I(1)$ is a contractible direct summand. Removing this contractible summand we obtain $F_sB_x^I\cong B_x^I(-1)\in {}^p\calK^{\geq 0}$.\qedhere
	\end{enumerate}
\end{proof}

\begin{cor}\label{Kgeq0}
	For any $x\in W^I$ we have $F_x^I\in {}^p\calK^{\geq 0}$.
\end{cor}
\begin{proof}
	This easily follows from Lemma \ref{RC1} since  $R_I\in{}^p\calK^{\geq 0}$ and $F_x^I\cong F_x\otimes R_I$ in $\calK^b(\sbim^I)$.
\end{proof}

\subsection{Singular Rouquier Complexes are Linear}\label{linear}

When Soergel's conjecture holds, we can describe quite explicitly the singular Rouquier complexes. (This explicit description is a crucial tool in \cite{Pat3} where the case of Grassmannians is studied in detail.)

\begin{lemma}\label{RC5}
Let $x\in W^I$.
For $i>0$ if ${}^iF_x^I$ contains a direct summand isomorphic to $B_z^I(j)$, then ${}^{i-1}F_x^I$ contains a direct summand isomorphic to $B_{z'}^I(j')$ with $z'>z$ and $j'<j$.
\end{lemma}
\begin{proof}
The proof is basically the same as in \cite[Lemma 6.11]{EW1}. From Theorem \ref{SConj} and \eqref{newch} we see that for any $y,z\in W^I$
the bimodule $\Gamma^I_{\geq z/>z}(B_y^I)$ is generated in degree $<\ell(z)$ if $y>z$ and $\Gamma^I_{\geq y/>y}(B_y^I)\cong R_{y,I}(-\ell(y))$.

The image of $B_z^I(j)$ in ${}^{i+1}F_x^I$ is contained in $\tau_{<-j}({}^{i+1}F_x^I)$ because of \eqref{homvanishing}: in fact any non-zero homomorphism in degree $0$ is an isomorphism and thus yields a contractible direct summand.

Applying $\Gamma^I_{\geq z/> z}$ to $F_x^I$ the direct summand $B_z^I(j)$ returns a summand $R_{z,I}(j-\ell(z))$.
This cannot be a direct summand in $\Gamma^I_{\geq z/> z} (\tau_{<-j}{}^{i+1}F_x^I)$ and cannot survive in the cohomology of the complex because of Lemma \ref{deltasplit}.
Thus $R_{z,I}(j-\ell(z))$ must be the image of a direct summand $R_{z,I}(j-\ell(z))$ in $\Gamma_{\geq z/>z}(\tau_{> -j}({}^{i-1}F_x))$. 

This implies that there is a direct summand $B_{z'}^I(j')$ in ${}^{i-1}F_x$ with $z'>z$ and $j'<j$.
\end{proof}

\begin{theorem}\label{RC4}
Let $x\in W^I$ and let $F_x^I$ be a singular Rouquier complex. Then:
\begin{enumerate}[i)]
\item ${}^0F_x^I=B_x^I$.
\item For $i\geq 1$, ${}^iF_x^I=\bigoplus (B_z^I(i))^{\oplus m_{z,i}}$ with $z<x$, $z\in W^I$ and $m_{z,i}\in \bbZ_{\geq 0}$.
\end{enumerate}
In particular, $F_x^I\in {}^p\calK^0$.
\end{theorem}
\begin{proof} One could use the same argument as in Lemma \ref{RC5} to deduce that since ${}^{-1}(F_x^I)=0$ and $\Gamma_{\geq x/>x}F_x^I$ we must have ${}^0(F_x^I)\cong B_x^I$.  By induction on $i$ we get ${}^iF_x^I=\tau_{\leq -i}F_x^I$ for any $i>0$. Now ii) follows since we already know $F_x^I\in {}^p\calK^{\geq 0}$ from Corollary \ref{Kgeq0}.
\end{proof}

\begin{remark}\label{invKLr}
 We can define the character of a complex $F\in \calK^b(\sbim^I)$ by
 \[\cha(F)=\sum_{i\in \bbZ}(-1)^i\cha({}^iF)\in \calH.\]
 If $x\in W^I$ and $\undx=s_1s_2\ldots s_k$ is a reduced expression we have 
 \[\cha(F_x^I)=\cha((F_{s_1}F_{s_2}\ldots F_{s_k})_I)=\bfH_x\undH_{I}=:\bfH_x^I.\]
 An immediate consequence of Theorem \ref{SConj} is that there is a non trivial morphism of degree $i$ between $B_x^I$ and $B_y^I$ for $x,y\in W^I$ only if $i$ and $\ell(x)-\ell(y)$ have the same parity. Therefore for all summands $B_y^I(i)\cus {}^iF_x^I$ the number $i-\ell(y)+\ell(x)$ is even. 
Because of Theorem \ref{RC4} we can write
\begin{equation}\label{invKL}
\bfH_x^I=\sum_{i\geq 0}(-1)^i\cha({}^iF_x)=\sum_{y\leq x} (-1)^{\ell(y)-\ell(x)} g^I_{y,z}(v)\undH_y^I
\end{equation}
 with $g_{x,x}(v)=1$ and $g_{y,x}(v)=\sum_{i>0} m_{y,i}v^i\in v\bbN[v]$. The polynomials $g^I_{x,y}$ are called the $I$-\emph{parabolic inverse Kazhdan-Lusztig polynomials}, and they are also determined by the following \emph{inversion formula}:
 \begin{equation}\label{invform}
\sum_{y\in W^I}(-1)^{\ell(y)-\ell(x)}g_{x,y}^I(v)h_{y,z}^I(v)=\delta_{x,z}.
 \end{equation}
One can use \eqref{invKL} to deduce that the $I$-parabolic inverse Kazhdan-Lusztig polynomials $g^I_{x,y}(v)$ have non-negative coefficients.
\end{remark}

By a dual argument we have  that ${}^iE_x^I(i)\cong {}^iF_x^I(-i)$ for all $i$, so in particular also $E_x^I\in {}^p\calK^0$.

By looking at the coefficient of $v$ in \eqref{invform} we see that $m_{z,1}$,
the coefficient of $v$ in $g_{z,x}^I(v)$, equals the coefficient of $v$ in $h_{z,x}^I$(v), hence they both coincide with $\dim \Hom^1(B_z^I,B_x^I)$.

We denote  by $d_x^{i}$, for $i<0$, the differentials in the complex $E_x^I$.

\begin{lemma}\label{degreeone}
Let $x\in W^I$. For any $z\in W^I$ fix a basis $\{\varphi^z_i\}_{i=1}^{m_{z,1}}$ of $\Hom^1(B_z^I,B_x^I)$ . Then there exists an isomorphism $K: (B_z^I(-1))^{\oplus m_{z,1}}\isom {}^{-1}E_x^I$ such that the following diagram
\begin{center}
\begin{tikzpicture}
	\node (a) at (0,0) {${}^{-1}E_x^I$};
	\node (b) at (3,0) {$\displaystyle \bigoplus_z (B_z^I(-1))^{\oplus m_{z,1}}$};
	\node (c) at (0,-3) {${}^0E_x^I=B_x^I$};

    \path[<-]  (a)edge node[above]{$\sim$} node[below] {$K$} (b);
	\path[->] (a) edge node[left] {$d_x^{-1}$} (c); 
	\path[->] (b) edge node[right] {$\quad\displaystyle \bigoplus_{z,i} \varphi^z_i$ } (c);	
 \end{tikzpicture}
\end{center}
commutes.
\end{lemma} 
\begin{proof}
Let $\displaystyle B=\bigoplus_z(B_z^I(-1))^{\oplus m_{z,1}}$ and consider the map $\displaystyle \varphi:=\bigoplus_{z,i} \varphi^z_i:B\raw B_x^I$. Then $\varphi$
induces a map of complexes concentrated in homological degree $0$
\[\varphi:F:=\bigoplus_{z,i} F_z^I(-1)^{\oplus m_{z,1}}\raw E_x^I.\]
From Lemma \ref{deltanabla} we see that $\varphi$ is homotopic to $0$. However, since ${}^1F$ is perverse, any map $K':{}^1F\raw B_x^I$ must be trivial, so there exists
$K: B\raw {}^{-1}E_x^I$
such that 
$d_x^{-1}\circ K = \varphi$.

Since $\{\varphi_i^z\}_i$ is a basis, the map $K$ cannot vanish on any direct summand of $B$. Notice that $K$ is of degree $0$, therefore $K$ is a split injection. Since $B\cong {}^{-1}E_x$ we conclude that $K$ is an isomorphism.
\end{proof}

\section{Hodge Theory of Singular Soergel Bimodules}
\label{sec:hodge}
Once we have Lemma \ref{Deltasplit} at disposal, we can adapt almost word by word the arguments of \cite{EW1} to the setting of singular Soergel bimodules.
As the proof of the results in this section are completely analogous to \cite{EW1} (but nevertheless rather long and technical) we do not carry out the details in this paper, but we refer to \cite[Chapter  4]{PatPhD} for exhaustive proofs.

We assume we are in the setting of Section \ref{linear}, so $\bbK=\bbR$ and Soergel's conjecture holds for $\frh^*$.
We denote by $(\frh^*)^I\subset \frh^*$ the subspace of $W_I$-invariants.
Let $\rho\in (\frh^*)^I\cug R^I$. We say that $\rho$ is \emph{ample} if $\rho(\alpha^\vee_s)>0$ for any $s\in S\setminus I$. Note that there exists such a $\rho$ with this property since the set $\{\alpha^\vee_s\}_{s\in S}$ is linearly independent in $\frh^*$.

\begin{theorem}[Hard Lefschetz Theorem for Singular Soergel Bimodules]\label{SHL}
	Let $\rho\in \frh^*$ ample. Then right multiplication by $\rho$ induces a degree $2$ map on $\bar{B_x^I}:=\bbR\otimes_R B_x^I$ such that, for any $i>0$ we have  an isomorphism
	\[ \rho^i: (\bar{B_x^I})^{-i}\raw (\bar{B_x^I})^i.\]
	Here $(\bar{B_x^I})^i$ denotes the degree $i$ component of $B_x^I$.
\end{theorem}

The indecomposable bimodules $B_x^I$ are self-dual, and moreover $\Hom(B_x^I,B_x^I)\cong \bbR$. This implies that there exists a unique (up to scalar) bilinear form 
\[\langle-,-\rangle_{B_x^I}: B_x^I\times B_x^I \raw R\]
such that for any $b,b'\in B_x^I$, $f\in R$ and $g\in R^I$ we have
$$\langle fb,b'\rangle_{B_x^I}=\langle b,fb'\rangle_{B_x^I}=f\langle b,b'\rangle_{B_x^I},$$
$$\langle bg,b'\rangle_{B_x^I}=\langle b,b'g\rangle_{B_x^I}.$$ 
Let $\rho\in (\frh^*)^I$ ample. Then we fix the sign by requiring that $\langle b,b\cdot \rho^{\ell(x)}\rangle_{B_x^I}>0$
for any $0\neq b\in (B_x^I)^{-\ell(x)}$.

The intersection form induces a real valued symmetric and $R^I$-invariant form $\langle-,-\rangle_{\bar{B_x^I}}$ on $\bar{B_x^I}$. For $i\geq 0$ we define the Lefschetz form 
\[(-,-)^{-i}_\rho=\langle-,-\cdot \rho^i\rangle_{\bar{B_x^I}}:\bar{B_x^I}^{-i}\times \bar{B_x^I}^{-i}\raw \bbR.\]

\begin{theorem}[Hodge-Riemann bilinear Relations for singular Soergel modules]\label{SHR}
	Let $x\in W^I$. For all $i\geq 0$ the restriction of Lefschetz form $(-,-)_\rho^{-i}$ to $P_{\rho}^{-i}=\ker(\rho^{i+1})\cug (\bar{B_x^I})^{-i}$ is $(-1)^{(-\ell(x)+i)/2}$-definite.
\end{theorem}

Theorem \ref{SHL} and Theorem \ref{SHR} have also consequences for non-singular bimodules, allowing us to extend the hard Lefschetz theorem and the Hodge-Riemann relations ``on the wall.''

Let $x\in W$ and $s\in S$ be  such that $xs>x$. Let $B_x\in \sbim$ be the corresponding indecomposable (non-singular) Soergel bimodule. Assume $I=\{s\}$, so that $w_I=s$. Then $(B_x)_I$ is a perverse singular Soergel bimodule, in fact we have:
\[\cha((B_x)_I)=\undH_x \undH_s=\undH_x^I+\sum_{\stackrel{ys>y}{y<x}}m_y\undH_y^I\qquad \text{ with }m_y\in \bbZ_{\geq 0}\]
%Moreover, all the $y$'s appearing in the sum have the same parity of $x$. 
We obtain the following:

\begin{cor}\label{semismallHL}
	Let $x\in W$ be such that $xs>x$. Then if $\rho\in (\frh^*)^s$ is ample, i.e. $\rho(\alpha_s^\vee)=0$ and $\rho(\alpha_t^\vee)>0$ for all $t\neq s$, multiplication by $\rho$ on $\bar{B_x}$ satisfies the hard Lefschetz theorem and the Hodge-Riemann bilinear relations.
\end{cor}
\begin{proof}
	Since 
	\begin{equation}\label{singnonsing}
	(B_x)_I\cong B_x^I\oplus \bigoplus_{\stackrel{ys>y}{y<x}} (B_y^I)^{\oplus m_y}
	\end{equation}
	hard Lefschetz for $\bar{B_x}$ follows from hard Lefschetz for all $y$ such that $B_y^I$ is  a direct summand in \eqref{singnonsing}. 
	
	Let $\varpi_s$ be a fundamental weight for $s$ and let $\rho_\zeta=\rho+\zeta\varpi_s$ for $\zeta\geq 0$. From the non-singular case, multiplication by $\rho_\zeta$ satisfies hard Lefschetz on $B_x$ for all $\zeta\geq 0$ and Hodge-Riemann for every $\zeta>0$ . Since the signature of a family of non-degenerate forms does not change, we deduce the  Hodge-Riemann bilinear relation for $\rho_0=\rho$.
\end{proof}

\begin{remark}
	Corollary \ref{semismallHL} has the following geometric motivation. Assume that $W$ is the Weyl group of a complex semisimple group $G$. Let $x\in W$ be such that $xs>x$ for $s\in S$ and let $X_x\cug G/B$ be the corresponding Schubert variety.
	Let $\bfP_s$ be the minimal parabolic subgroup of $G$ containing $s$.  
	Then the restriction of the projection $G/B\raw G/\bfP_s$ to $X_x$ is semismall. It follows from \cite[Theorem 2.3.1]{dCM3} that the pull-back of any ample class on $G/\bfP_s$ satisfies hard Lefschetz and Hodge-Riemann on $X_x$.
\end{remark}

\begin{remark}\label{newSConj}
	We can obtain from Corollary \ref{semismallHL} an alternative proof of Soergel's conjecture, that translates more closely de Cataldo and Migliorini's proof of the decomposition theorem in \cite{dCM3}.
	
	Assume $w\in W$ such that $ws>w$ and assume that $\cha(B_x)=\undH_x$ for all $x<ws$. Let $I=\{s\}$ and fix $\rho$ ample in $(\frh^*)^s$. Let $x< w\in W$ be such that
	$xs>x$.  Consider the primitive subspace 
	\[ P_\rho^{-k}:= \ker (\rho^{k+1}) \cap (\bar{B_w})^{-k}.\] 
	
	We have a symmetric form 
	\[(-,-):\Hom(B_x^I,(B_w)_I)\times \Hom(B_x^I,(B_w)_I)\raw \End(B_x^I)\cong \bbR\] defined by $(f,g)=g^*\circ f$, where $g^*$ denotes the map adjoint to $g$ with respect of the intersection forms.
	Then we can show, as in \cite[Theorem 4.1]{EW1}, that the map 
	\[\iota:\Hom(B_x^I,(B_w)_I)\raw P_\rho^{-\ell(x)}\]
	defined by $f\mapsto f(c_{\text{bot}})$ is injective. Moreover, if we equip $P_\rho^{-\ell(x)}$ with the Lefschetz form, then $\iota$ is an isometry (up to a positive scalar) and, by the Hodge-Riemann bilinear relations, the form $(-,-)$ is definite on $\Hom(B_x^I,(B_w)_I)$. If $d=\dim \Hom(B_x^I,(B_w)_I)$, it follows that $(B_x^I)^{\oplus d}$ is a direct summand of $(B_w)_I$, hence $(B_{xs})^{\oplus d}$ is a direct summand of $B_wB_s$.
\end{remark}

\begin{example}
	Let $W$ be the Weyl group of type $A_3$ with simple reflections labeled $s,t,u$. Let $I=\{s,t\}$, so that $w_I=sts$. Then $stu\in W^I$ but a simple computation in the Hecke algebra shows that 
	$$\undH_{stu}\undH_{sts}=\undH_{stu}^I+\undH_{u}^I+
	(v+v^{-1})\undH_{id}^I.$$
	Therefore, the singular Soergel bimodule $(B_{stu})_I$ is not perverse, and there is no $\rho\in (\frh^*)^I$ which satisfies hard Lefschetz on $\bar{B_{stu}}$.
\end{example}

\bibliographystyle{alpha}

\end{document}